\documentclass{article}
\usepackage[utf8]{inputenc}

\usepackage{graphicx,graphics}

\usepackage{rotating}
\usepackage[twoside]{geometry}
\usepackage{epsfig}
\usepackage{cmap}

\usepackage{graphicx,graphics}

\usepackage{hyperref}
\hypersetup{
pdftitle={paper},            
colorlinks=true,      
linkcolor=black,         
citecolor=blue}

\usepackage{amssymb,amsfonts,amstext,amsthm}
\usepackage{mathrsfs}
\usepackage[intlimits]{amsmath}
\usepackage{comment} 

\newtheorem{definition}{\bf Definition}[section]
\newtheorem{proposition}{\bf Proposition}[section]

\newtheorem{lemma}{\bf Lemma}[section]
\newtheorem{theorem}{\bf Theorem}[section]

\newtheorem{remark}{Remark}[section]

\newcounter{SE}

\title{Eigenfunction correlators under power-law SULE and localization for lattice operators}
\author{M. Aloisio, C. R. de Oliveira, R. Matos, D. Oliveira and M. Pigossi}

\begin{document}

\maketitle
\sloppy
\begin{abstract}
We develop a deterministic framework showing that a power-law form of semi-uniform localization of eigenfunctions (SULE) imposes strong structural constraints on lattice operators,  with consequences of both spectral and dynamical nature. For instance, as spectral consequences we prove that power-law SULE yields geometric constraints on localization centers (such as their equidistribution) and quantitative bounds on eigenfunction correlators. As a dynamical consequence we obtain power-law localization in the sense of finite $q$-moments (up to a certain power $q$) of the position operator. Conversely, suitable bounds on eigenfunction correlators imply a corresponding form of power-law SULE, establishing a close connection between these notions. This highlights the role of power-law SULE as a structural mechanism governing localization beyond the exponential regime, including features typically associated with random operators, such as Anderson-type models. Our results reveal that power-law localization is intrinsically geometric: the spatial distribution of localization centers directly influences eigenfunction correlators and transport properties. As an application, we obtain power-law localization for long-range lattice operators with Stark-type potentials of sublinear growth whose spectral regime exhibits asymptotically collapsing spectral gaps and quasi-resonant structures, without relying on perturbative methods. Applications to long-range random operators are also discussed.
\end{abstract}

\ 


\

\noindent{\bf  AMS classification codes}: 28A80 (primary), 42A85 (secondary).    

\renewcommand{\thetable}{\Alph{table}}


\newpage

\section{Introduction}

\subsection{Contextualization}

The phenomenon of dynamical localization in quantum dynamics for random models with finite-range hopping has long been a classical subject of investigation, dating back to Anderson's seminal work \cite{And58}. From a mathematically rigorous perspective, tools such as the fractional moment method \cite{Aizenman,Aizenman2}, multiscale analysis \cite{Spencer,Klein-VonD,Germ-Klein}, and the semi-uniform localization of eigenfunctions (SULE) \cite{DamanikFillman1,DJLS,Tcheremchantsev} form a robust framework for this setting.

While these models describe a broad class of physical systems, situations involving long-range hopping with power-law decay arise naturally in contexts such as dipolar Frenkel excitons, nuclear spin systems, and the quantum Kepler model \cite{AltshulerL1997,Shi2}. Recently, numerous studies have uncovered localization phenomena in systems governed by random \cite{Disertori,Sun2,Sunpartin,Shi1,Shi2,Shi4,Shi6}, quasi-periodic \cite{Shi5,Shi55}, and Stark operators \cite{Aloisio,Sun} with power-law hopping. Beyond their theoretical interest, these mechanisms are motivated by their connection to the quantum suppression of chaos \cite{Shi2} and Stark localization \cite{Oliveira01,Pigossi1,Pigossi,Roeck,Hu2,Hu,Nazareno,Sun3}.

To quantify dynamical localization, let $d \in \mathbb{N}$ and let $H$ be a self-adjoint operator in $\ell^2(\mathbb{Z}^d)$. For wave packet solutions $\psi(t) = e^{-itH}\psi$ of the Schr\"odinger equation, spreading is measured through moments of the position operator. We say that $H$ exhibits dynamical localization if, for all initial conditions defined by the canonical basis $\{\delta_k\}$, each moment is uniformly bounded in time, i.e.,
\begin{equation*}
\sup_{t \in \mathbb{R}} \sum_{n \in \mathbb{Z}^d} \langle n \rangle^q  |\langle e^{-itH}\delta_k ,\delta_n \rangle|^2 < \infty, \, \,  \forall q > 0\,,
\end{equation*}
where $\langle n \rangle = \|n\|+1$, with $\|\cdot\|$ denoting the maximum norm on $\mathbb{Z}^d$. While the RAGE Theorem \cite{Oliveira} implies that dynamical localization requires pure point spectrum, the converse is not generally true \cite{SMDTB,DJLS,Germinet}.

In the exponential-scale setting, SULE provides a deterministic framework to guarantee dynamical localization \cite{DJLS}. An operator $H$ exhibits SULE if it possesses an orthonormal basis of eigenfunctions $\{\varphi_m\}_{m \in \mathbb{Z}^d}$ and associated localization centers $j_m \in \mathbb{Z}^d$ such that, for some $\alpha > 0$ and any $\varepsilon > 0$,
\begin{equation*}
|\varphi_m(n)| \leq \gamma_\varepsilon e^{\varepsilon \|j_m\|-\alpha\|n-j_m\|}.
\end{equation*}
If $\varepsilon = 0$ and $j_m = m$ one has the stricter Uniform Localization of Eigenfunctions (ULE), observed in Stark operators \cite{Pigossi}, though generally precluded in ergodic models due to resonances \cite{Damanik,DamanikFillman1,DamanikFillman2}.

For long-range operators with power-law decay, it is natural to expect corresponding power-law dynamical estimates \cite{Aizenman2,Aloisio,Disertori,Sun2,Kerner,Kraisler,Shi2,Shi4,Shi5,Shi6,Sun}. However, a deterministic  mechanism linking power-law eigenfunction decay, the geometry of localization centers, and correlator bounds remains poorly understood compared to probabilistic frameworks relying on disorder.

In this work, we bridge this gap by developing a  non-perturbative abstract framework based on power-law SULE (Definition \ref{SULEdef}). Unlike in the exponential regime, where spatial irregularities of localization centers are absorbed by the rapid decay \cite{DJLS}, polynomial decay cannot suppress the cumulative effects of spatial disorder. Consequently, (uniform) power-law localization is  geometric, that is,  properties such as the asymptotic distribution of localization centers directly dictate eigenfunction correlators and transport estimates.

Our main contributions are as follows:
\begin{enumerate}
\item We establish that power-law localization is structurally intertwined with a suitable notion of power-law SULE (Theorem \ref{mainthm} and Remark \ref{mainremark}). Our results reveal that  power-law localization is intrinsically geometric: the spatial distribution of localization centers directly influences eigenfunction correlators and transport properties (Remark \ref{mainremark} (ii)).
\item We prove the rigidity of localization centers under certain structural assumptions (particularly for some long-range random models), showing they become asymptotically equidistributed (Theorem \ref{mainthm} and Section \ref{aplication}).
\item For Stark-type potentials with sublinear growth, we prove power-law localization without relying on KAM methods or Green's function estimates. 
\end{enumerate}

The remainder of this note is organized as follows. Subsection \ref{subsecmainreult} presents our main results. Section \ref{aplication} details applications, Section \ref{secproof} contains the proof of the main theorem, and Section \ref{proofaplic} handles proofs related to Stark-type potentials.  We denote by $\ell^2(\mathbb{Z}^d)$ and $\ell^\infty(\mathbb{Z}^d)$ the standard sequence spaces with their usual norms. $I$ denotes the identity operator, and $\{\delta_n\}_{n \in \mathbb{Z}^d}$ is the canonical basis in $\ell^2(\mathbb{Z}^d)$. For $m \in \mathbb{Z}^d$, we set $\|m\| := \displaystyle\max_{1 \leq j \leq d} |m_j|$ and denote $\langle m \rangle := \|m\| + 1$.

\subsection{Main Result}\label{subsecmainreult}

\begin{definition}\label{SULEdef}
{\rm Let $H$ be a linear operator on $\ell^2(\mathbb{Z}^d)$, $\alpha>0$ and $\beta \geq 0$. We say that $H$ exhibits $(\alpha,\beta)$-Power-Law SULE if there exists an orthonormal basis $\{\varphi_m\}_{m \in \mathbb{Z}^d}$ of eigenfunctions of $H$ such that, for each $m \in \mathbb{Z}^d$, there exists a localization center $j_m \in \mathbb{Z}^d$ satisfying
\[
|\varphi_m(n)| \le C_{\alpha,\beta}\, \frac{\langle j_m \rangle^{\beta}}{\langle n - j_m \rangle^{\alpha}} \, \,  \text{for all } n \in \mathbb{Z}^d,
\]
where $C_{\alpha,\beta}$ is uniform in $m$ and $n$.}
\end{definition}

We now present our main result.
\begin{theorem}\label{mainthm}
Let \(H\) be a linear operator on \(\ell^2(\mathbb Z^d)\). Assume that \(H\) exhibits
\((\alpha,\beta)\)-Power-Law SULE  and let \(\{j_m\}_{m\in\mathbb Z^d}\) be the associated localization centers.

\begin{enumerate}
\item[\rm(i)] \textbf{Upper density bound for localization centers and power-law localization.}
Assume that
\begin{equation}\label{eqmain1010102}  
 \alpha>\beta+\frac d2 .   
\end{equation}
(a)  Then, for every
\(L\geq 1\), the set
\[
\{m\in\mathbb Z^d:\langle j_m\rangle\leq L\}
\]
is finite and
\[
\limsup_{L\to\infty}
\frac{\#\{m:\langle j_m\rangle\leq L\}}{(2L+1)^d}
\leq 1 .
\]

(b) For every \(\varepsilon>0\), there exists a constant
\(\gamma_{\alpha,\beta,\varepsilon}>0\) such that, for all \(k,n\in\mathbb Z^d\),
\[
\sum_{m\in\mathbb Z^d}|\varphi_m(k)|\,|\varphi_m(n)|
\leq
\gamma_{\alpha,\beta,\varepsilon}
\frac{\min\{\langle k\rangle,\langle n\rangle\}^{\beta+\frac{d}{2}+\varepsilon}}
{\langle n-k\rangle^{\alpha-\beta-\frac{d}{2}-\varepsilon}} .
\]
(c) If, in addition, $H$ is self-adjoint and $\alpha>\beta+d,$ then, for every \(k\in\mathbb Z^d\) and every
\[
0<q<2\alpha-2\beta-2d,
\]
one has
\[
\sup_{t\in\mathbb R}
\sum_{n\in\mathbb Z^d}
\langle n\rangle^q
\left|\langle e^{-itH}\delta_k,\delta_n\rangle\right|^2
<\infty .
\]

\item[\rm(ii)] \textbf{Rigidity and equidistribution of localization centers.}
Assume that
\begin{equation}\label{eqmain1010103}
\alpha>\beta+d.
\end{equation}
Then
\[
\lim_{L\to\infty}
\frac{\#\{m:\langle j_m\rangle\leq L\}}{(2L+1)^d}
=1 .
\]
\end{enumerate}
\end{theorem}

\begin{remark}\label{mainremark}
\end{remark}

\begin{enumerate}
\item[(i)] \emph{Equivalence.} Assume that there exists a family of functions $\{\varphi_j\}_{j \in \mathbb{Z}^d}$ satisfying the correlator bound $\sum_{j} |\varphi_j(k)|\,|\varphi_j(n)| \le \gamma_{\alpha,\beta} \langle k \rangle^{\beta} / \langle n-k \rangle^{\alpha}$. By choosing a suitable center $j_m\in\mathbb Z^d$, one immediately recovers power-law SULE, i.e., for every $\varepsilon>0$ there exists a constant $C_{\alpha,\beta,\varepsilon}>0$ such that
\[
|\varphi_m(n)| \le C_{\alpha,\beta,\varepsilon}\, \frac{\langle j_m \rangle^{\beta+\frac{d}{2}+\varepsilon}}{\langle n-j_m \rangle^{\alpha}}.
\]
Combined with Theorem \ref{mainthm}(ii), this establishes an equivalence: power-law SULE is an unavoidable deterministic mechanism governing power-law localization.

\item[(ii)] \emph{H\"older exponents and sharpness.}  The conditions in \eqref{eqmain1010102} and \eqref{eqmain1010103} indicate that a larger semi-uniform decay exponent $\alpha$ (relative to dimension $d$) permits a larger error exponent $\beta$ for spatial disorder. This delicate interplay between decay and geometry has no parallel in the exponential case \cite{Aloisio,Sun2}. We note that the examples discussed in~\cite{Aizenman,deMoura,Disertori} seem to suggest that the condition \(\alpha>\beta+\frac{d}{2}\) in \eqref{eqmain1010102} may be optimal.

\item[(iii)] The proof of Theorem \ref{mainthm} relies entirely on orthogonality identities and volumetric counting in $\mathbb{Z}^d$ to establish the equidistribution of centers. A key ingredient is the use of conjugate H\"older exponents \(p\) and \(p'\) to balance the geometric error exponent \(\beta\) against the spatial decay exponent \(\alpha\). This structural control enables the spatial decompositions required for the correlator bounds, which subsequently yield the uniform dynamical moments.
\end{enumerate}


\section{Applications}\label{aplication}

\subsection{Long-range Stark-type lattice operators with sublinear growth}

Let $a \in \ell^1_r(\mathbb{Z})$ ($r \geq 0$), i.e., $\|a\|_r = \sum_{m \in \mathbb{Z}} |a(m)||m|^r < \infty$, and suppose that $a(0) = 0$ and $a(m) = a^*(-m)$ for all $m \in \mathbb{Z}$. Define, for each $\eta>0$,
\[
(T_\eta u)(n) = \sum_{m \in \mathbb{Z}} a(n-m) u(m) + \operatorname{sgn}(n)|n|^\eta u(n),
\]
on the domain
\[
\text{dom } T_\eta := \left\{u\in \ell^2(\mathbb{Z}) : \sum_{n\in \mathbb{Z}}|n|^{2\eta}|u(n)|^2 < \infty\right\}.
\]

In contrast with the linear potential case $\eta=1$ studied in \cite{Aloisio}, the present setting allows for accumulation of eigenvalues at infinity for $0<\eta<1$, leading to the presence of quasi-resonances. In this case, while the linear growth model studied in \cite{Aloisio} exhibits a power-law ULE structure, the present model only admits a power-law SULE structure.

\begin{theorem} \label{thmperturbation}
Let $0<\eta<1$ and let $T_\eta$ be as above and consider $T_\eta + b$, where $b \in \ell^\infty(\mathbb{Z}, \mathbb{R})$. Then:
\begin{enumerate}
\item[{\rm (i)}] $T_\eta + b$ has purely discrete spectrum and there exists $\gamma > 0$ such that  the eigenvalues satisfy
\[
\bigl|\lambda_n - \operatorname{sgn}(n)\,|n|^\eta\bigr| \le \gamma, \, \,  n \in \mathbb{Z}.
\]
\item[{\rm (ii)}] If $a \in \ell_r^1(\mathbb{Z})$, then any orthonormal basis of eigenvectors $\{\varphi_m\}$ satisfies
\[
|\varphi_m(n)| \le \gamma_{r,\eta}\, \frac{\langle m \rangle^{(r+1)(1-\eta)}}{\langle m - n \rangle^{(r+1)\eta}}.
\]
\item[{\rm (iii)}] For every $k \in \mathbb{Z}$ and for all $0 < q < (r+1)(4\eta - 2) - 2$,
\[
\sup_{t \in \mathbb{R}} \sum_{n \in \mathbb{Z}} \langle n\rangle^q \, 
\big|\langle e^{-it(T_\eta + b)}\delta_k, \delta_n \rangle\big|^2 < \infty.
\]
\end{enumerate}
\end{theorem}

The proof follows the general strategy of \cite{Aloisio}, relying on spectral asymptotics rather than KAM-type arguments or Green's function techniques (see Section \ref{proofaplic} ahead). Next, we discuss a specific example.

\subsubsection{Fractional Laplacian with sublinear Stark potential}

Let $0<s<1$ and $0<\eta<1$. Consider the one-dimensional operator
\[
H = \Delta^s + \operatorname{sgn}(\cdot)\,|\cdot|^\eta + b
\]
acting on $\ell^2(\mathbb{Z})$, where $\Delta^s$ denotes the discrete fractional Laplacian of order $s$ \cite{Gebert} and $b \in \ell^\infty(\mathbb{Z},\mathbb{R})$ is an arbitrary bounded perturbation.

In this case, recall that we can write $(\Delta^s u)(n) = \sum_{m \in \mathbb{Z}} a(n-m) u(m), n \in \mathbb{Z},$ with $a \in \ell^1_{r-\epsilon}(\mathbb{Z})$ for all $\epsilon>0$, where $r = 2s$ \cite{Gebert}. Theorem~\ref{thmperturbation} implies that the second moment ($q=2$) is finite provided
\[
s > \displaystyle\frac{1}{2} \, \,  \text{and} \, \,  
\eta > \frac{1}{2} + \frac{1}{(2s+1)}.
\]

The spectral structure described in Theorem \ref{thmperturbation} suggests the presence of quasi-resonances whenever $0<\eta<1$. Indeed, in this case, the eigenvalues satisfy
\[
|\lambda_{n+1}-\lambda_n|\sim |n|^{\eta-1},
\]
so that consecutive eigenvalues become asymptotically clustered at high energies.

From a heuristic perspective, this accumulation mechanism may be related to a transition between localization and transport regimes in the $(s,\eta)$-plane, similarly to what is observed numerically in some long-range Anderson-type models \cite{deMoura}.

\subsection{Long-range hopping random operators}

We revisit the study of power-law localization for long-range hopping random operators considered in \cite{Aizenman,Disertori,Sun2}. The main novelty here concerns the spatial distribution of localization centers. More precisely, we show that the localization centers associated with power-law localized eigenfunctions satisfy a strong asymptotic equidistribution property (see~(\ref{eqDistriCentRand})), ruling out large-scale concentration phenomena and persistent gaps. For simplicity, we restrict to the one-dimensional case, although the discussion extends to higher dimensions.

Let $r > 2$ and let $H_\omega$ be a random operator on $\ell^2(\mathbb{Z})$ of the form
\[
H_\omega = \lambda^{-1} T_a + V_\omega,
\]
where $T_a$ is defined by $(T_au)(n)=\sum_{m \in \mathbb{Z}} a(n-m) u(m)$ with
\[
a(m-n) = |m-n|^{-r} \, \,  \text{for } m \neq n, \, \,  \text{and } a(0) = 0.
\]
The potential $V_\omega(n) = \omega_n$ is a sequence of i.i.d.\ random variables in $[0,1]^{\mathbb{Z}}$ with a Lipschitz common distribution.

It was shown in \cite{Aizenman,Disertori,Sun2} that for every sufficiently large $r$ and sufficiently large disorder $\lambda$, for almost every $\omega$, the operator $H_\omega$ has pure point spectrum and admits an orthonormal basis of eigenfunctions $\{\varphi_m^\omega\}_{m \in \mathbb{Z}}$ with localization centers $j_m^\omega \in \mathbb{Z}$ satisfying
\[
|\varphi_m^\omega(n)| \leq C_{\omega,r} \frac{\langle j_m^\omega \rangle^\beta}{\langle n - j_m^\omega \rangle^\alpha}, \, \,  n \in \mathbb{Z},
\]
for some $\alpha(r),\beta(r)>0$. In particular, the operator exhibits $(\alpha,\beta)$-Power-Law SULE. Under this assumption, Theorem~\ref{mainthm} yields that for every $k$ and for all 
\begin{equation}\label{eq056565}
0 < q < 2 \alpha - 2\beta-2,    
\end{equation}
the corresponding dynamical moment is uniformly bounded,
\[
\sup_{t \in \mathbb{R}} \sum_{n \in \mathbb{Z}} \langle n \rangle^q \, 
\big| \langle e^{-itH_\omega} \delta_k, \delta_n \rangle \big|^2 < \infty.
\]
We note that, in many cases of interest, one has $2 \alpha - 2\beta-2>0,$ see \cite{Aizenman,Disertori,Sun2}. In particular, condition
\eqref{eqmain1010103} is satisfied. Although the boundedness of such dynamical moments is already known in this setting, it is worth emphasizing that, under the same assumptions on \(r\), our result yields a larger range of finite moments than the one obtained in \cite{Sun2}. Moreover, our approach applies to situations which are more general than those covered in \cite{Aizenman,Disertori}.

Regarding the localization centers $\{j_m^\omega\}_{m \in \mathbb{Z}}$, randomness allows for local fluctuations. However, Theorem~\ref{mainthm} imposes constraints on their global distribution, implying that they are asymptotically equidistributed, in the sense that
\begin{equation}  \label{eqDistriCentRand}
\lim_{L \to \infty} \frac{\#\{m : \langle j_m^\omega \rangle \leq L\}}{2L+1} = 1.
\end{equation}
This shows that large-scale concentration or persistent gaps in the distribution of centers are ruled out.


\section{Proof of Theorem \ref{mainthm}}\label{secproof}

Next, we prove Theorem \ref{mainthm}. To this end, some preliminary results are required. The next lemma shows that, under the power-law SULE assumption, the mass of an eigenfunction outside a sufficiently large neighborhood of its localization center is small and exhibits power-law decay.

Since $\langle j_m\rangle \geq 1$, if $H$ satisfies $(\alpha,\beta')$-Power-Law SULE with $\beta'=0$, then $H$ also satisfies $(\alpha,\beta)$-Power-Law SULE for all $\beta>0$. Therefore, throughout this section, we may assume without loss of generality that $\beta > 0$.

\begin{lemma} \label{lema7.2} Let $\alpha>\beta + \displaystyle\frac{d}{2}$ and let $p\in\left(\frac{\alpha}{\alpha-\beta},\frac{2\alpha}{d}\right)$ and  $p'$ its conjugate exponent. Suppose that $H$ exhibits $(\alpha,\beta)$-Power-Law SULE. Then, for every $\varepsilon>0$, there exists a constant $D_{\varepsilon,\alpha,\beta}$ such that, for each sufficiently large $r>1$, each $m \in \mathbb{Z}^d$ and each $L \geq 1$,
\[
\sum_{\langle n - j_m \rangle \,\geq\, \varepsilon (\langle j_m \rangle + L)} 
|\varphi_m(n)|^2 
\;\leq\; 
D_{\varepsilon,\alpha,\beta} \, L^{-2\alpha/ r p'}\, \langle j_m \rangle^{-(2\alpha/p)+d}.
\]
\end{lemma}

\begin{proof} Since $\alpha>\beta+\frac d2$, we have
\[
\frac{\alpha}{\alpha-\beta}<\frac{2\alpha}{d}.
\]
Hence, we may choose
\[
p\in\left(\frac{\alpha}{\alpha-\beta},\frac{2\alpha}{d}\right),
\]
and let $p'$ be its conjugate exponent. Then
\[
\alpha>\frac{pd}{2}
\, \, \text{and}\, \, 
\beta<\alpha\left(1-\frac1p\right)=\frac{\alpha}{p'}.
\]

By the Power-Law SULE condition, for all $m,n \in \mathbb{Z}^d$, 
\begin{equation}\label{eq01010101}
|\varphi_m(n)| \;\leq\; C_{\alpha,\beta} \, 
\frac{\langle j_m \rangle^\beta}{\langle n - j_m \rangle^\alpha}.    
\end{equation}

Since $\beta < \alpha/p'$, one has $\alpha/(\beta p')>1$.  Let $1<r'<\alpha/(\beta p')$ and let $r>1$ such that $1/r+1/r'=1$. Note that: $\alpha/(r'p') - \beta > 0$.

Assume that $\langle n - j_m \rangle\geq \varepsilon(\langle  j_m \rangle+L)$. Using the Young's inequality $a^{r'} + b^{r} \geq ab$ for $a,b \geq 0$, with $a=\langle j_m \rangle^{1/r'}$ and $b=L^{1/r}$, we have $\langle j_m \rangle+L \geq \langle j_m \rangle^{1/r'} L^{1/r}$. Thus,
\begin{align*}
\langle n - j_m \rangle 
&= \langle n - j_m \rangle^{1/p} \langle n - j_m \rangle^{1/p'} \\ 
&\geq \varepsilon^{1/p'} \langle n - j_m \rangle^{1/p} \, (\langle j_m \rangle+L)^{1/p'} \\ 
&\geq \varepsilon^{1/p'} \langle n - j_m \rangle^{1/p} \, \langle j_m \rangle^{1/(r'p')} L^{1/(rp')}.
\end{align*}
Since  $\alpha/(r'p') - \beta > 0$, substituting this lower bound into \eqref{eq01010101} yields
\[
|\varphi_m(n)| \;\leq\; 
\frac{C_{\alpha,\beta}}{\langle n - j_m \rangle^{\alpha/p} (L^{1/r} \cdot \varepsilon)^{\alpha/p'}}.
\]
Squaring this expression and summing over all $n$ such that $\langle n - j_m \rangle\geq \varepsilon(\langle  j_m \rangle+L)$, we obtain
\[
\sum_{\langle n - j_m \rangle \geq \varepsilon(\langle  j_m \rangle+L)} |\varphi_m(n)|^2
\;\leq\; C_{\alpha,\beta}^2 \, L^{-2\alpha/rp'}\, 
\varepsilon^{-2\alpha/p'}
\sum_{\langle n - j_m \rangle\geq \varepsilon(\langle j_m \rangle+L)} \frac{1}{\langle n - j_m \rangle^{2\alpha/p}}.
\]
Writing $k = n - j_m$ and using the fact that $L \geq 1$, we can bound the sum by
\begin{align*}
\sum_{\langle n - j_m \rangle \geq \varepsilon(\langle j_m \rangle+L)} |\varphi_m(n)|^2
&\leq C_{\alpha,\beta}^2 \, L^{-2\alpha/rp'}\, \varepsilon^{-2\alpha/p'}
\sum_{\langle k \rangle\geq \varepsilon\langle j_m \rangle} \frac{1}{\langle k \rangle^{2\alpha/p}}\\ 
&\leq D_{\varepsilon,\alpha,\beta} \, L^{-2\alpha/rp'}\, \langle j_m \rangle^{-(2\alpha/p)+d},  
\end{align*}
for some constant $D_{\varepsilon,\alpha,\beta}$. The last inequality holds because the sum converges and is bounded by an integral of order $R^{-(2\alpha/p)+d}$, which is finite since $\alpha > pd/2$. This proves the lemma.
\end{proof}

\subsection{Proof of Theorem \ref{mainthm} (i) (a)}

Since Theorem~1.1 (i) (a) will be used throughout the proof, we restate it as the following lemma, which we prove below.

\begin{lemma}\label{lemma3}
Let $\alpha > \beta + \displaystyle\frac{d}{2}$. Suppose that $H$ exhibits $(\alpha,\beta)$-Power-Law SULE. Then, for each $L \geq 1$, the set 
\[
\{ m : \langle j_m \rangle \leq L \}
\]
is finite, and its cardinality satisfies
\[
\limsup_{L\to\infty}
\frac{\#\{m:\langle j_m \rangle\le L\}}{(2L+1)^d}
\ \le\ 1.
\]
\end{lemma}

\begin{proof} 
Recall the two basic orthogonality identities for the orthonormal basis $\{\varphi_m\}$:
\begin{equation}\label{eq:72a}
\sum_{n\in\mathbb Z^d} |\varphi_m(n)|^2 = 1 \, \,  \text{for each $m \in\mathbb Z^d$}, 
\end{equation}
and
\begin{equation}\label{eq:72b}
\sum_{m\in\mathbb Z^d} |\varphi_m(n)|^2 = 1 \, \,  \text{for each $n\in\mathbb Z^d$}. 
\end{equation}

Let $\varepsilon>0$ and assume that $\langle j_m \rangle\le L$. If $\langle n \rangle\ge (1+2\varepsilon)L$, then by the triangle inequality,
\[
\langle n - j_m \rangle \ge \langle n \rangle - \langle j_m \rangle \ge (1+2\varepsilon)L - L = 2\varepsilon L \ge \varepsilon(\langle j_m \rangle+L).
\]
By Lemma \ref{lema7.2}, it follows that, for some sufficiently large $r>1$,
\begin{eqnarray*}
\sum_{\langle n \rangle\ge (1+2\varepsilon)L} |\varphi_m(n)|^2 &\leq& \sum_{\langle n - j_m \rangle\ge \varepsilon(\langle j_m \rangle+L)} |\varphi_m(n)|^2 \leq D_{\varepsilon,\alpha,\beta}\,L^{-2\alpha/rp'}\,\langle  j_m \rangle^{-(2\alpha/p)+d}\\ &\leq& D_{\varepsilon,\alpha,\beta}\,L^{-2\alpha/rp'},    
\end{eqnarray*}
where we used $\langle j_m \rangle \geq 1$ and $-(2\alpha/p)+d < 0$. Using \eqref{eq:72a}, it follows that the mass inside the ball is bounded below by
\begin{equation*}
\sum_{\langle n  \rangle\le (1+2\varepsilon)L} |\varphi_m(n)|^2 
\ \ge\ 1 - D_{\varepsilon,\alpha,\beta}\,L^{-2\alpha/rp'}.
\end{equation*}
We now sum over all $m$ such that $\langle j_m \rangle\le L$. Using \eqref{eq:72b} to exchange the order of summation, we get
\begin{align*}
(2(1+2\varepsilon)L+1)^d 
&\geq \sum_{\langle n \rangle\le (1+2\varepsilon)L} 1 \\
&= \sum_{\langle n \rangle\le (1+2\varepsilon)L} \sum_{m \in \mathbb{Z}^d} |\varphi_m(n)|^2\\
&\geq \sum_{\langle n \rangle\le (1+2\varepsilon)L} \sum_{\{m:\langle j_m \rangle\le L\}} |\varphi_m(n)|^2\\
&= \sum_{\{m:\langle j_m \rangle\le L\}} \sum_{\langle n \rangle\le (1+2\varepsilon)L} |\varphi_m(n)|^2\\
&\geq \#\{m:\langle j_m \rangle\le L\}\,\big(1-D_{\varepsilon,\alpha,\beta} L^{-2\alpha/rp'}\big).    
\end{align*}
Since the left-hand side is finite, $\#\{m:\langle j_m \rangle\le L\}$ is finite. Dividing by $(2L+1)^d$ and taking the limit superior yields
\begin{equation*}
\limsup_{L\to\infty}
\frac{\#\{m:\langle j_m \rangle\le L\}}{(2L+1)^d}
\ \le\ (1+2\varepsilon)^d.
\end{equation*}
Because $\varepsilon > 0$ is arbitrary, it immediately follows that
\begin{equation}\label{eq:73}
\limsup_{L\to\infty}
\frac{\#\{m:\langle j_m \rangle\le L\}}{(2L+1)^d}
\ \le\ 1.
\end{equation}
In particular, there exists $\gamma >0$ (independent of $L$) such that
\begin{equation} \label{contagemjm}
\#\{m:\langle  j_m \rangle\le L\} \ \le\ \gamma\, L^d \, \,  \text{for all } L\ge1.
\end{equation}
\end{proof}

We have the following result, which is a direct consequence of the Lemma \ref{lemma3} and is of independent interest.

\begin{proposition}\label{prop:summability-centers} Let $\alpha>\beta+\frac d2$ and suppose that $H$ exhibits
$(\alpha,\beta)$-Power-Law SULE with localization centers
$\{j_m\}_{m\in\mathbb Z^d}$. Then, for every $\varepsilon>0$, one has
\[
A_\varepsilon:=\sum_{m\in\mathbb Z^d}
\frac{1}{\langle j_m\rangle^{d+\varepsilon}}<\infty.
\]
\end{proposition}

\begin{proof}
Fix $\varepsilon>0$. For each $\ell\geq 0$, define the dyadic shell
\[
S_\ell:=\left\{m\in\mathbb Z^d:2^\ell\leq \langle j_m\rangle<2^{\ell+1}\right\}.
\]
By the Lemma \ref{lemma3}, we have
\[
\#S_\ell
\leq \#\left\{m\in\mathbb Z^d:\langle j_m\rangle<2^{\ell+1}\right\}
\leq \gamma 2^{(\ell+1)d}.
\]
Therefore, reorganizing the series into dyadic shells, we obtain
\begin{align*}
\sum_{m\in\mathbb Z^d}
\frac{1}{\langle j_m\rangle^{d+\varepsilon}}
&=
\sum_{\ell=0}^{\infty}
\sum_{m\in S_\ell}
\frac{1}{\langle j_m\rangle^{d+\varepsilon}}  \le
\sum_{\ell=0}^{\infty}
\sum_{m\in S_\ell}
\frac{1}{2^{\ell(d+\varepsilon)}} \le
\sum_{\ell=0}^{\infty}
\frac{\#S_\ell}{2^{\ell(d+\varepsilon)}} \\
&\le
\sum_{\ell=0}^{\infty}
\frac{\gamma 2^{(\ell+1)d}}{2^{\ell(d+\varepsilon)}} =
\gamma 2^d
\sum_{\ell=0}^{\infty}
2^{-\ell\varepsilon}
<\infty.
\end{align*}
\end{proof}

\subsection{Proof of Theorem \ref{mainthm} (ii)}

Note that since $\alpha>\beta+d$, we have
\[
\frac{\alpha}{\alpha-\beta}<\frac{\alpha}{d}.
\]
Hence, we may choose
\[
p\in\left(\frac{\alpha}{\alpha-\beta},\frac{\alpha}{d}\right),
\]
and let $p'$ be its conjugate exponent. Then
\[
\alpha>pd
\, \, \text{and}\, \, 
\beta<\alpha\left(1-\frac1p\right)=\frac{\alpha}{p'}.
\]

To establish the full limit, let $0<\varepsilon <1$ and consider the following spatial decomposition based on the localization centers 
\[
A_L:=\Big\{m:\ \langle j_m \rangle\le \frac{1+\varepsilon}{1-\varepsilon}\,L\Big\},
\, \, 
B_L:=\Big\{m:\ \langle j_m \rangle> \frac{1+\varepsilon}{1-\varepsilon}\,L\Big\}.
\]
For every $m\in B_L$ and every $n$ satisfying $\langle n \rangle\le L$, we have $\langle n - j_m \rangle \ge \varepsilon(\langle  j_m \rangle+L)$. Hence, by Lemma \ref{lema7.2}, for every $m\in B_L$,
\[
\sum_{\langle n  \rangle\le L} |\varphi_m(n)|^2 \ \le\ D_{\varepsilon,\alpha,\beta}\,\,\langle  j_m \rangle^{-(2\alpha/p)+d}.
\]
Summing this over all $m\in B_L$, we obtain
\[
\sum_{m\in B_L}\sum_{\langle n\rangle\le L}
|\varphi_m(n)|^2
\le
D_{\varepsilon,\alpha,\beta}
\sum_{\ell=0}^{\infty}
\sum_{2^\ell L\le \langle j_m\rangle <2^{\ell+1}L}
\langle j_m\rangle^{-2\alpha/p+d}.
\]
Since $\langle j_m\rangle\ge 2^\ell L$ in the $\ell$-th shell, we can bound the inner sum by
\[
\sum_{2^\ell L\le \langle j_m\rangle <2^{\ell+1}L}
\langle j_m\rangle^{-2\alpha/p+d}
\le
\#\{m:\langle j_m\rangle<2^{\ell+1}L\}
(2^\ell L)^{-2\alpha/p+d}.
\]
From Lemma \ref{lemma3}, we know that there exists $\gamma>0$ such that
\[
\#\{m:\langle j_m\rangle<R\}\le \gamma R^d
\]
for all $R>0$. Therefore,
\[
\#\{m:\langle j_m\rangle<2^{\ell+1}L\}
\le
\gamma (2^{\ell+1}L)^d.
\]
Hence,
\begin{align*}
\sum_{2^\ell L\le \langle j_m\rangle <2^{\ell+1}L}
\langle j_m\rangle^{-2\alpha/p+d}
&\le
\gamma (2^{\ell+1}L)^d
(2^\ell L)^{-2\alpha/p+d} \\
&=
\gamma 2^d
L^{2d-2\alpha/p}
2^{\ell(2d-2\alpha/p)}.
\end{align*}
Inserting this into the bound for $B_L$ gives
\[
\sum_{m\in B_L}\sum_{\langle n\rangle\le L}
|\varphi_m(n)|^2
\le
\gamma D_{\varepsilon,\alpha,\beta} 2^d
L^{2d-2\alpha/p}
\sum_{\ell=0}^{\infty}
2^{\ell(2d-2\alpha/p)}.
\]
The last series converges because
\[
2d-\frac{2\alpha}{p}<0,
\]
that is, 
\[
\alpha>pd.
\]
Therefore, there exists $\gamma'>0$ such that
\begin{equation}\label{eq004}
\sum_{m\in B_L}\sum_{\langle n\rangle\le L}
|\varphi_m(n)|^2
\le
\gamma' D_{\varepsilon,\alpha,\beta}
L^{2d-2\alpha/p}.    
\end{equation}
Finally, exploiting the completeness of the basis in the ball of radius $\lfloor L - 1 \rfloor $,
\[
 (2\lfloor L - 1 \rfloor +1)^d =  
 \sum_{m \in \mathbb{Z}^d}\sum_{\langle n  \rangle\le L} |\varphi_m(n)|^2
\le \#A_L + \sum_{m\in B_L}\sum_{\langle n  \rangle\le L} |\varphi_m(n)|^2.
\]
Dividing by $(2L+1)^d$, using \eqref{eq004}, and noting that the exponent of $L$ is strictly negative, taking the limit inferior as $L \to \infty$ yields
\[
1 \ \le\ \liminf_{L\to\infty}
\frac{\#A_L}{(2L+1)^d}.
\]
By making the change of variables $L' = \displaystyle\frac{1+\varepsilon}{1-\varepsilon} L$ and recalling the limit superior from Lemma \ref{lemma3}, we obtain, for every $\varepsilon>0$,
\[
\left(\frac{1-\varepsilon}{1+\varepsilon} \right)^d \leq  \liminf_{L\to\infty} \frac{\#\{m:\ \langle j_m \rangle\le L\}}{(2L+1)^d} \leq \limsup_{L\to\infty} \frac{\#\{m:\ \langle j_m \rangle\le L\}}{(2L+1)^d} \leq 1.
\]
Letting $\varepsilon \to 0$ completes the proof.
\hfill \qedsymbol


\subsection{Proof of Theorem \ref{mainthm} (i) (b)}

Note that we only need to prove the result for sufficiently small $\varepsilon > 0$. Let $r:=\frac12\langle n-k\rangle .$ Note that, for every $m\in\mathbb Z^d$,
\[
\langle j_m-n\rangle\geq r
\, \, \text{or}\, \, 
\langle j_m-k\rangle\geq r.
\]
Set
\[
\mathcal E_n:=\{m:\langle j_m-n\rangle\geq r\},
\, \, 
\mathcal E_k:=\{m:\langle j_m-k\rangle\geq r\}.
\]
Then, $\mathbb Z^d=\mathcal E_n\cup\mathcal E_k$.

We estimate the term over $\mathcal E_n$. By Cauchy-Schwarz, since the eigenfunctions are normalized, we have
\[
\sum_{m\in\mathcal E_n}|\phi_m(k)|\,|\phi_m(n)|
\leq
\left(\sum_{m\in\mathcal E_n}|\phi_m(n)|^2\right)^{1/2}.
\]
By Power-Law SULE,
\[
|\phi_m(n)|^2
\leq
C_{\alpha,\beta}^2
\frac{\langle j_m\rangle^{2\beta}}
{\langle j_m-n\rangle^{2\alpha}}.
\]
Moreover,
\[
\langle n\rangle
\leq
\langle k\rangle+\langle n-k\rangle
=
\langle k\rangle+2r
\leq
4\langle k\rangle r.
\]
Since $m\in\mathcal E_n$, $r\leq\langle j_m-n\rangle$, and therefore
\[
\langle j_m\rangle
\leq
\langle n\rangle+\langle j_m-n\rangle
\leq
5\langle k\rangle\langle j_m-n\rangle .
\]
Thus,
\[
\frac{\langle j_m\rangle^{2\beta}}
{\langle j_m-n\rangle^{2\alpha}}
\leq
5^{d+2\beta+2 \varepsilon}
\frac{\langle k\rangle^{d+2\beta+2 \varepsilon}}
{\langle j_m\rangle^{d+2 \varepsilon}
\langle j_m-n\rangle^{2\alpha-d-2\beta-2 \varepsilon}}.
\]
Hence, using $\langle j_m-n\rangle\geq r$ and
Proposition~\ref{prop:summability-centers},
\[
\sum_{m\in\mathcal E_n}|\phi_m(n)|^2
\leq
C_{\alpha,\beta}^2 5^{d+2\beta+2 \varepsilon}
\frac{\langle k\rangle^{d+2\beta+2 \varepsilon}}
{r^{2\alpha-d-2\beta-2 \varepsilon}}
\sum_{m \in \mathbb{Z}^d}\frac1{\langle j_m\rangle^{d+2 \varepsilon}}
\leq
A_\epsilon C_{\alpha,\beta}^2  5^{d+2\beta+2 \varepsilon}
\frac{\langle k\rangle^{d+2\beta+2 \varepsilon}}
{r^{2\alpha-d-2\beta-2 \varepsilon}}.
\]
Thus,
\[
\sum_{m\in\mathcal E_n}|\phi_m(k)|\,|\phi_m(n)|
\lesssim 
\frac{\langle k\rangle^{\frac d2+\beta+\varepsilon}}
{r^{\alpha-\beta-\frac d2-\varepsilon}},
\]
where $r= \frac{1}{2}\langle n-k\rangle$.

The estimate over $\mathcal E_k$ is identical, replacing $n$ by $k$ and using
\[
\langle j_m\rangle
\leq
2 \langle k\rangle\langle j_m-k\rangle .
\]

Finally, note that the same argument, with \(k\) and \(n\) interchanged, yields the same estimates with \(\langle n\rangle\) in place of \(\langle k\rangle\). Therefore, the result follows.
\hfill \qedsymbol

\subsection{Proof of Theorem \ref{mainthm} (i) (c)}

Assume that $H$ is self-adjoint. Using the spectral functional calculus \cite{Oliveira} and Parseval's identity, we can expand the unitary evolution in the orthonormal eigenbasis $\{\varphi_m\}_{m \in \mathbb{Z}^d}$. For any $k, n \in \mathbb{Z}^d$, this yields
\begin{eqnarray*}
\langle e^{-itH} \delta_k, \delta_n \rangle
&=& \sum_{m \in \mathbb{Z}^d} \langle \delta_k, \varphi_m \rangle \langle e^{-itH} \varphi_m, \delta_n \rangle \\
&=& \sum_{m \in \mathbb{Z}^d} e^{-it\lambda_m} \langle \delta_k, \varphi_m \rangle \langle \varphi_m, \delta_n \rangle.
\end{eqnarray*}

Taking the absolute value and applying the triangle inequality, we obtain the uniform bound
\[
\big|\langle e^{-itH} \delta_k, \delta_n \rangle\big|
\leq \sum_{m \in \mathbb{Z}^d} |\langle \delta_k, \varphi_m \rangle| \, |\langle \varphi_m, \delta_n \rangle|
= \sum_{m \in \mathbb{Z}^d} |\varphi_m(k)| \, |\varphi_m(n)|.
\]

Squaring this inequality, multiplying by the weight $\langle n \rangle^q$, and summing over $n \in \mathbb{Z}^d$, the uniform boundedness of the $q$-th moment follows directly from the correlator estimate established in Theorem~\ref{mainthm} (i) (b). The specified upper bound for $q$ ensures that the resulting series converges over the spatial lattice.
\hfill \qedsymbol


\section{Proof of Theorem \ref{thmperturbation}}\label{proofaplic}

In order to prove Theorem \ref{thmperturbation}, we will use the following result, whose proof coincides with that of Theorem 2.2 in \cite{Aloisio}, and is therefore omitted.  

\begin{theorem}\label{assintopticsthm}
Let $\eta>0$ and let $H = A + V$ on $\ell^2(\mathbb{Z})$, where $A$ is bounded  self-adjoint  and
\[
(Vu)(n) = \mathrm{sgn}(n)|n|^\eta\,u(n), \, \,  n \in \mathbb{Z}.
\]
Then $H$ is self-adjoint with purely discrete spectrum $\{\lambda_n\}_{n\in\mathbb{Z}}$. In this case, one writes
\begin{equation*}
\dots\lambda_{-n}, \dots, \lambda_{-2} \leq \lambda_{-1} \leq \lambda_0 \leq \lambda_1 \leq \lambda_2, \dots, \lambda_n  \dots 
\end{equation*}
where $\lambda_0 = \displaystyle\min \{\lambda_n \mid \,  \lambda_n \geq 0 \}$. Moreover,
\[
|\lambda_n - \mathrm{sgn}(n)|n|^\eta| \le ||A||, \, \,  n \in \mathbb{Z}.
\]
\end{theorem}

\begin{lemma} \label{teclemma004}
Let $\eta\in(0,1)$ and define, for each $m\in\mathbb{Z}$,
\[
V_\eta(m)=\mathrm{sgn}(m)\,|m|^\eta.
\]
Then, for all $m,n\in\mathbb{Z}$ with $m\neq n$ and $(m,n)\neq(0,0)$,
\begin{equation*}
\frac{1}{|V_\eta(m)-V_\eta(n)|}
\leq
\frac{C_\eta\,\langle m \rangle^{1-\eta}}{\langle m - n \rangle^\eta},
\end{equation*}
where $C_\eta=2^{\eta+1}\eta^{-1}$.
\end{lemma}

\begin{proof} Note that for $n = 0$ or $m = 0$ the result is immediate. Hence, we assume that $m, n \neq 0$.

We first consider the case in which $m$ and $n$ have the same sign.

\noindent\textbf{Case 1: $\mathrm{sgn}(m)=\mathrm{sgn}(n)=+$}.  
Assume that $m>n$. Since $x^\eta$ is differentiable on $(0,\infty)$, the Mean Value Theorem yields a point $z\in(n,m)$ such that
\[
m^\eta-n^\eta=\eta z^{\eta-1}(m-n).
\]
As $\eta-1<0$ and $z<m$, we have $z^{\eta-1}\ge m^{\eta-1}$. Moreover, since $m-n\ge1$ and $x\ge x^\eta$ for $x\ge1$, it follows that
\[
m^\eta-n^\eta
\ge
\eta m^{\eta-1}(m-n) \ge
\eta m^{\eta-1}(m-n)^\eta.
\]
Hence,
\begin{eqnarray*}
\frac{1}{|V_\eta(m)-V_\eta(n)|} &=& \frac{1}{\bigl||m|^\eta-|n|^\eta\bigr|}
\le
\frac{\eta^{-1}|m|^{1-\eta}}{|m-n|}\\
&\le&
\frac{2^\eta \eta^{-1} (|m|+1)^{1-\eta}}{(|m-n|+1)^\eta}= \frac{2^\eta \eta^{-1}\,\langle m \rangle ^{1-\eta}}{\langle m -n \rangle ^\eta}.    
\end{eqnarray*}
If instead $n>m$, the previous estimate gives
\[
\frac{1}{\bigl||m|^\eta-|n|^\eta\bigr|}
\le
\frac{\eta^{-1}|n|^{1-\eta}}{|m-n|}.
\]
Using $|n|\le |m|+|n-m|$ together with the inequality $(x+y)^p\le x^p+y^p$ for $p\in(0,1)$, we obtain
\[
|n|^{1-\eta}
\le
|m|^{1-\eta}+|n-m|^{1-\eta}.
\]
Since $|n-m|^{\eta}\le |n-m|$ for $|n-m|\ge1$, it follows that
\begin{eqnarray*}
\frac{1}{|V_\eta(m)-V_\eta(n)|} &=& \frac{1}{\bigl||m|^\eta-|n|^\eta\bigr|}
\le
\frac{\eta^{-1} (|m|^{1-\eta}+1)}{|m-n|^\eta}\\ &\le& \frac{2^{\eta+1} \eta^{-1}\,\langle m \rangle ^{1-\eta}}{\langle m -n \rangle ^\eta}.    
\end{eqnarray*}

\noindent\textbf{Case 2: $\mathrm{sgn}(m)=\mathrm{sgn}(n)=-$}.  
If $m>n$, then $0<-m<-n$ and the estimate follows from the previous case by symmetry. The situation $n>m$ is analogous.

\noindent\textbf{Case 3: $\mathrm{sgn}(m)\neq\mathrm{sgn}(n)$}.  
In this case,
\[
|V_\eta(m)-V_\eta(n)|=|m|^\eta+|n|^\eta,
\, \, 
|m-n|=|m|+|n|.
\]
Using again $(x+y)^p\le x^p+y^p$, $p\in(0,1)$, we obtain
\[
|m-n|^\eta
\le
|m|^\eta+|n|^\eta
=
|V_\eta(m)-V_\eta(n)|.
\]
Therefore,
\[
\frac{1}{|V_\eta(m)-V_\eta(n)|}
\le
\frac{1}{|m-n|^\eta}
\le
 \frac{2^{\eta+1} \eta^{-1}\,\langle m \rangle ^{1-\eta}}{\langle m -n \rangle ^\eta}.
\]
This completes the proof.
\end{proof}

\subsection{Proof of Theorem \ref{thmperturbation}}

Part (i) follows from Theorem \ref{assintopticsthm} by taking $A = T_a + b$. Part (iii) follows from Theorem~\ref{mainthm} combined with part (ii). Therefore, it remains only to prove part (ii). 

Let $\{\varphi_m\}_{m\in\mathbb{Z}}$ be an orthonormal basis of $\ell^2(\mathbb{Z})$ consisting of eigenfunctions of $T_\eta + b$, that is,
\[
(T_\eta + b)\varphi_m = \lambda_m \varphi_m, \, \,  m \in \mathbb{Z}.
\]

Since $b \in \ell^\infty(\mathbb{Z})$, by part (i), there exists $\gamma > 0$ such that
\[
\bigl|\lambda_m - \operatorname{sgn}(m)|m|^\eta - b(n)\bigr| \le \gamma,
\, \,  \text{for all } m,n \in \mathbb{Z}.
\]

Let $V_\eta(m) = \operatorname{sgn}(m)|m|^\eta$ and let $r \geq 0$. Assume first that
\[
|V_\eta(m)-V_\eta(n)| \le 2\gamma.
\]
Then, by Lemma~\ref{teclemma004}, for $m \neq n$ and $(m,n) \neq (0,0)$, we can then bound the eigenfunction by its $\ell^2$ norm,
\[
|\varphi_m(n)| \le 1
\le \frac{(2\gamma)^{r+1}}{|V_\eta(m)-V_\eta(n)|^{r+1}}
\le \frac{(2\gamma)^{r+1} C_\eta^{r+1}\, \langle m \rangle^{(r+1)(1-\eta)}}{\langle m - n \rangle^{(r+1)\eta}},
\]
where $C_\eta = 2^{\eta+1}\eta^{-1}$. When $m = n$ or $(m,n) = (0,0)$, we directly have
\[
|\varphi_m(n)| \le 1
\le \frac{\langle m \rangle^{(r+1)(1-\eta)}}{\langle m - n \rangle^{(r+1)\eta}}.
\]
Therefore, we only need to estimate the case where the potential difference is large, namely,
\[
|V_\eta(m)-V_\eta(n)| > 2\gamma.
\]

Note that the eigenvalue equation
\[
(T_\eta + b)\varphi_m(n) = \lambda_m \varphi_m(n)
\iff
\bigl(T_a + \operatorname{sgn}(n)|n|^\eta + b(n)\bigr)\varphi_m(n)
= \lambda_m \varphi_m(n)
\]
can be rewritten to isolate $\varphi_m(n)$ resulting in
\[
\varphi_m(n)
=
\frac{(T_a\varphi_m)(n)}
{\lambda_m - (\operatorname{sgn}(n)|n|^\eta + b(n))}
=
\frac{\sum_{k\in\mathbb{Z}} a(k)\varphi_m(n-k)}
{\lambda_m - \operatorname{sgn}(n)|n|^\eta - b(n)}.
\]
Taking absolute values and applying the triangle inequality to the denominator yields
\[
|\varphi_m(n)|
\le
\frac{\sum_{k\in\mathbb{Z}} |a(k)|\,|\varphi_m(n-k)|}
{\bigl|\operatorname{sgn}(n)|n|^\eta - \operatorname{sgn}(m)|m|^\eta \bigr|  - \bigl|\lambda_m - \operatorname{sgn}(m)|m|^\eta - b(n)\bigr|}.
\]
Since we assumed $|V_\eta(m)-V_\eta(n)| > 2\gamma$, the denominator is bounded from below by
\[|V_\eta(m)-V_\eta(n)| - \gamma > \frac{1}{2}|V_\eta(m)-V_\eta(n)|.\] 
Thus,
\begin{eqnarray}\label{maineq0101}
\nonumber |\varphi_m(n)|
&\le&
\frac{\sum_{k\in\mathbb{Z}} |a(k)|\,|\varphi_m(n-k)|}
{|\operatorname{sgn}(m)|m|^\eta - \operatorname{sgn}(n)|n|^\eta| - \gamma}
\\
&\le&
\frac{2\sum_{k\in\mathbb{Z}} |a(k)|\,|\varphi_m(n-k)|}
{|\operatorname{sgn}(m)|m|^\eta - \operatorname{sgn}(n)|n|^\eta|}.
\end{eqnarray}

From this point onward, we divide the proof into cases. First, we prove the result for each \(r \in \mathbb{N} \cup \{0\}\).

\ 

\noindent \textbf{Case 1 :} r $\in \mathbb{N} \cup \{0\}$.  We now proceed by induction on $r$. 
\vspace{0.5em}

\noindent \textbf{Base Case ($r=0$):} 
If $a \in \ell^1_0(\mathbb{Z})$, by applying Lemma \ref{teclemma004} to \eqref{maineq0101} and bounding $|\varphi_m(n-k)| \le 1$, we conclude there exists a constant $\gamma_{0,\eta} > 0$, depending only on $\eta$ and $\|a\|_0$, such that
\[
|\varphi_m(n)|
\le
\frac{2 C_\eta \sum_{k\in\mathbb{Z}} |a(k)|\,\langle m \rangle^{1-\eta}}
{\langle m - n \rangle^\eta} \le \gamma_{0,\eta}\,
\frac{\langle m \rangle^{1-\eta}}{\langle m - n \rangle^\eta}.
\]

\vspace{0.5em}
\noindent \textbf{Induction Step:} 
Assume that the result holds for some $r \geq 0$, meaning that if $a \in \ell^1_{r}(\mathbb{Z})$, then
\[
|\varphi_m(n)|
\le \gamma_{r,\eta}\,
\frac{\langle m \rangle^{(r+1)(1-\eta)}}
{\langle m - n \rangle^{(r+1)\eta}}.
\]
We will prove that the same estimate holds for $r+1$ if $a \in \ell^1_{r+1}(\mathbb{Z})$.

Let $a \in \ell^1_{r+1}(\mathbb{Z}) \subset \ell^1_{r}(\mathbb{Z})$. Consider a shift $k \in \mathbb{Z}$ and suppose that $|(m-n)+k| \ge 2$ and $|k| \ge 2$. Using the fact that $A+B \le AB$ for $A,B \ge 2$, we have
\begin{equation}\label{eq010102}
|m-n|
\le |(m-n)+k| + |-k|
\le |(m-n)+k|\,|k|
\iff
\frac{1}{|(m-n)+k|}
\le \frac{|k|}{|m-n|}.
\end{equation}
By the induction hypothesis and \eqref{eq010102}, for $|(m-n)+k| \ge 2$ and $|k| \ge 2$, we obtain
\begin{eqnarray*}
|\varphi_m(n-k)|
&\le& \gamma_{r,\eta}\,
\frac{\langle m \rangle^{(r+1)(1-\eta)}}
{\langle m-n+k \rangle^{(r+1)\eta}} \\
&\le&  \gamma_{r,\eta}\,
\frac{\langle m \rangle^{(r+1)(1-\eta)}}
{|m-n+k|^{(r+1)\eta}} \\
&\le&  \gamma_{r,\eta}\,
\frac{\langle m \rangle^{(r+1)(1-\eta)}}
{|m-n|^{(r+1)\eta}}\,|k|^{(r+1)\eta} \\
&\le&  2^{(r+1)\eta} \gamma_{r,\eta}\,
\frac{\langle m \rangle^{(r+1)(1-\eta)}}
{\langle m-n \rangle^{(r+1)\eta}}\,|k|^{r+1}.
\end{eqnarray*}

Since $a(0)=0$, we can substitute this bound into \eqref{maineq0101} and apply Lemma~\ref{teclemma004} to factor out the potential difference, yielding
\[
|\varphi_m(n)|
\le
\frac{2 C_\eta \,\langle m \rangle^{1-\eta}}
{\langle m - n \rangle^\eta}
\sum_{k\in\mathbb{Z}} |a(k)|\,|\varphi_m(n-k)|.
\]
We split the sum into the bulk terms where the induction bound applies, and the boundary terms corresponding to small $|k|$ or small $|(m-n)+k|$, that is,
\begin{eqnarray}\label{maineq07}
\nonumber |\varphi_m(n)|
&\le&
\frac{2^{(r+1)\eta+1} \gamma_{r,\eta} C_\eta \,\langle m \rangle^{(r+2)(1-\eta)}}
{\langle m - n \rangle^{(r+2)\eta}}
\sum_{\substack{|(m-n)+k|\ge 2\\ |k|\ge 2}}
|a(k)|\,|k|^{r+1} \\
\nonumber&+&
\frac{2 C_\eta \,\langle m \rangle^{1-\eta}}
{\langle m - n \rangle^\eta}
|a(\pm 1)|\,|\varphi_m(n \mp 1)| \\
\nonumber &+&
\frac{2 C_\eta \,\langle m \rangle^{1-\eta}}
{\langle m - n \rangle^\eta}
|a(n-m)|\,|\varphi_m(m)| \\
&+&
\frac{2 C_\eta \,\langle m \rangle^{1-\eta}}
{\langle m - n \rangle^\eta}
|a(n-m\pm 1)|\,|\varphi_m(m \mp 1)|.
\end{eqnarray}

Since $a \in \ell_{r+1}^1(\mathbb{Z})$, the sequence decays sufficiently fast, and there exists a constant $\gamma'_r > 0$ such that for all $k \in \mathbb{Z}$,
\begin{equation}\label{eqfn}
|a(k)| \le \gamma'_r \frac{1}{\langle k \rangle^{r+1}}
\le \gamma'_r \frac{1}{\langle k \rangle^{(r+1)\eta}}.    
\end{equation}

Finally, evaluate each part of \eqref{maineq07}. The first sum converges and is bounded by $\|a\|_{r+1}$; the second term is controlled by applying the induction hypothesis to $\varphi_m(n \mp 1)$; and the third and fourth terms are controlled by applying the decay bound \eqref{eqfn} to $a(n-m)$ and $a(n-m \pm 1)$. Combining these contributions, we conclude that there exists a constant $\widetilde{\gamma_{r,\eta}} > 0$, depending only on $\eta$, $r$, and $\|a\|_{r+1}$, such that
\[
|\varphi_m(n)|
\le \widetilde{\gamma_{r,\eta}}
\frac{\langle m \rangle^{(r+2)(1-\eta)}}
{\langle m - n \rangle^{(r+2)\eta}}.
\]
This completes the induction step and the proof of the Case 1.

\ 

\noindent \textbf{Case 2 : \(r\in[0,1).\)} Set
\[
M:=\langle m\rangle,
\, \, 
R:=\langle m-n\rangle.
\]
We claim that there exists \(C_{r,\eta}>0\) such that
\begin{equation}\label{eqmain01012}
|\phi_m(n)|
\le
C_{r,\eta}
\frac{M^{(r+1)(1-\eta)}}{R^{(r+1)\eta}}.
\end{equation}

Note that we may assume that \(R \geq 4\), since the case \(R \leq 4\) follows from $|\phi_m(n)| \leq 1.$

Note also that if \(R^\eta\le M^{1-\eta}\), then \eqref{eqmain01012} follows from \(|\phi_m(n)|\le1\). Hence, we may assume that
\begin{equation}\label{eqmain010123}
R^\eta>M^{1-\eta}.
\end{equation}

Recall that, by \eqref{maineq0101}, we have
\begin{equation}\label{maineq010133}
|\phi_m(n)|
\le
C_{r,\eta}
\frac{M^{1-\eta}}{R^\eta}
\sum_{k\in\mathbb Z}
|a(k)|\,|\phi_m(n-k)|.
\end{equation}
We split the sum into the regions \(|k|\le R/2\) and \(|k|>R/2\).

For \(|k|\le R/2\), one has
\[
\langle m-(n-k)\rangle\gtrsim R.
\]
Using the estimate corresponding to \(r=0\), it follows that
\[
|\phi_m(n-k)|
\lesssim
\frac{M^{1-\eta}}{R^\eta}.
\]
Therefore, the contribution of this region to the right-hand side of \eqref{maineq010133} is bounded by
\[
C_{r,\eta}
\frac{M^{2(1-\eta)}}{R^{2\eta}}
=
C_{r,\eta}
\frac{M^{(r+1)(1-\eta)}}{R^{(r+1)\eta}}
\left(
\frac{M^{1-\eta}}{R^\eta}
\right)^{1-r}.
\]
Since \(r<1\), condition \eqref{eqmain010123} implies that this term is bounded by
the right-hand side of \eqref{eqmain01012}.

For \(|k|>R/2\), using \(|\phi_m(n-k)|\le1\) and \(a\in\ell^1_r(\mathbb Z)\), we obtain
\[
\sum_{|k|>R/2}
|a(k)|\,|\phi_m(n-k)|
\le
\sum_{|k|>R/2}|a(k)|
\lesssim
R^{-r}\|a\|_r.
\]
Hence, the contribution of this region is bounded by
\[
C_{r,\eta}
\frac{M^{1-\eta}}{R^{r+\eta}}.
\]
Since
\[
r+\eta\ge (r+1)\eta
\, \, \text{and}\, \, 
M^{1-\eta}\le M^{(r+1)(1-\eta)},
\]
this term is also bounded by the right-hand side of \eqref{eqmain01012}. This proves \eqref{eqmain01012} for every \(r\in[0,1)\).

\ 

\noindent \textbf{Case 3 : $r \in \mathbb{R}$ and $r \geq 0.$} We write
\[
r=N+\theta,
\, \, 
N\in\mathbb N\cup\{0\},
\, \, 
\theta\in[0,1).
\]
The previous argument yields the estimate for \(\theta\). Since the
induction step from \(s\) to \(s+1\) does not require \(s\) to be an
integer, applying it successively for
\[
s=\theta,\theta+1,\ldots,\theta+N-1
\]
as in Case 2, yields the desired estimate for \(r=\theta+N\). Therefore, the result follows. 
\hfill \qedsymbol


\begin{center} \Large{Acknowledgments} 
\end{center}
\addcontentsline{toc}{section}{Acknowledgments}

\noindent M. Aloisio was supported by grant \#2025/25338-1 from the São Paulo Research Foundation (FAPESP) and in part by grant \#01/24/APQ-03132-24 from the Minas Gerais Research Foundation (FAPEMIG). C. R. de Oliveira thanks CNPq (a Brazilian government agency) for partial support under grant 303689/2021-8. R. Matos thanks CNPq for partial support under grants 402952/2023-5 and  402249/2024-0. D. Oliveira thanks CAPES (a Brazilian government agency) for partial support.


\ 

\noindent  \noindent Moacir Aloisio. Email: moacir.aloisio@ufvjm.edu.br, DME, UFVJM, Diamantina, MG, 39100-000, Brazil and ICMC, USP, S\~ao Carlos, SP, 13566-590, Brazil

\noindent  C\'esar R. de Oliveira. Email: oliveira@ufscar.br,  DM,   UFSCar, S\~ao Carlos, SP, 13560-970 Brazil

\noindent Rodrigo Matos. Email: rodrigo-matos@puc-rio.br, DM, PUC-Rio, Rio de Janeiro, RJ, 22451-900 Brazil

\noindent  Deberton Oliveira. Email: moura.deberton@gmail.com, DM, UFMG, Belo Horizonte, MG, 31270-901 Brazil

\noindent  Mariane Pigossi. Email: mariane.pigossi@ufes.br, DMAT, UFES, Vit\'oria, ES, 29075-910 Brazil


\begin{thebibliography}{15}
\addcontentsline{toc}{section}{References}

\bibitem{Aizenman} M. Aizenman and S. Molchanov, Localization at large disorder and at extreme energies: an elementary derivation. Commun. Math. Phys. {\bf 157} (1993), 245--278.

\bibitem{Aizenman2} M. Aizenman and S. Warzel, Random Operators. Disorder Effects on Quantum Spectra and Dynamics, GSM 168, AMS (2015).

\bibitem{Aloisio} M. Aloisio, Dynamical localization and eigenvalue asymptotics: long-range hopping lattice operators with electric field. Ann. Henri Poincaré (2026). https://doi.org/10.1007/s00023-026-01698-9

\bibitem{SMDTB} M. Aloisio, S. L. Carvalho, C. R. de Oliveira, Spectral Measures and Dynamics: Typical Behaviors.  Berlin, Springer Nature (2023). 

\bibitem{AltshulerL1997} B. L. Altshuler and L. S. Levitov,  Weak chaos in a quantum Kepler problem. Phys. Rep. {\bf 288} (1997),  487--512.

\bibitem{And58} P. W. Anderson, Absence of diffusion in certain random lattices. Phys. Rev. {\bf 109} (1958), 1492--1505.

\bibitem{Damanik} D. Damanik,  Schr\"odinger operators with dynamically defined potentials, Ergodic Theory Dynam. Systems. {\bf 37} (2017), 1681-1764

\bibitem{DamanikFillman1} D. Damanik and J. Fillman,  One-dimensional ergodic Schrödinger operators, I. general theory, Graduate Studies in Mathematics {\bf 221}, American Mathematical Society, (2022).

\bibitem{DamanikFillman2} D. Damanik and J. Fillman, One-dimensional ergodic Schrödinger operators, II. specific classes, Graduate Studies in Mathematics {\bf 249}, American Mathematical Society, (2024).

\bibitem{deMoura} F. A. B. F. de Moura and M. L. Lyra, Delocalization in the 1D Anderson model with long-range correlated disorder. Phys. Rev. Lett. {\bf 81} (1998), 3735.

\bibitem{Oliveira} C. R. de Oliveira, Intermediate spectral theory and quantum dynamics. Progress in Math. Phys. Basel, Birkh\"auser, (2009).

\bibitem{Oliveira01} C. R. de Oliveira and M. Pigossi, Exact solvable family of discrete Schr\"odinger operators with long-range hoppings. Math. Scand. {\bf 130} (2024), 527--537.

\bibitem{Pigossi1} C. R. de Oliveira and M. Pigossi, Point spectrum and SULE for time-periodic perturbations of discrete 1D Schr\"odinger operators with electric fields. J. Stat. Phys. {\bf 173} (2018), 140--162.

\bibitem{Pigossi} C. R. de Oliveira and M. Pigossi, Proof of dynamical localization for perturbations of discrete 1D Schr\"odinger operators with uniform electric fields. Math. Z. {\bf 291} (2019), 1525--1541.

\bibitem{DJLS} R. Del Rio, S. Jitomirskaya, Y. Last and B. Simon, Operators with singular continuous spectrum, IV. Hausdorff dimension, rank one perturbations and localization. J. Anal. Math. {\bf 69} (1996), 153--200.

\bibitem{Roeck} W. De Roeck, A. Hannani, A. Lerose and N. Vandenbosch, Stark localization of interacting particles. Preprint. arXiv:2602.23352 [math-ph] (2026).

\bibitem{Disertori} M. Disertori, R. Maturana Escobar and C. Rojas-Molina, Decay of the Green’s function of the fractional Anderson model and connection to long-range SAW. J. Stat. Phys. {\bf 191} (2024), article number 33.

\bibitem{Spencer} J. Fr\"ohlich and T. Spencer,  Absence of diffusion in the Anderson tight binding model for large disorder or low energy. Commun. Math. Phys.  {\bf 88} (1983), 151--184.

\bibitem{Gebert} M. Gebert and C. Rojas-Molina, Lifshitz tails for the fractional Anderson model. J. Stat. Phys. {\bf 179} (2020), 341--353.

\bibitem{Germinet} F.  Germinet, A. Kiselev  and S. Tcheremchantsev, Transfer matrices and transport for Schr\"odinger operators. Ann. Inst. Fourier. {\bf 54} (2004), 787--830.
\bibitem{Germ-Klein} F. Germinet and A. Klein, Bootstrap multiscale analysis and localization in random media. (2001) Comm. Math. Phys. {\bf 222}, 415-448.

\bibitem{Hu2} S. Hu and Y. Sun, Localized state for nonlinear disordered stark model. Preprint. arXiv:2603.09243 [math.DS] (2026).

\bibitem{Hu} S. Hu and Y. Sun, Wannier-Stark Localization for time quasi-periodic Hamiltonian operator on $\mathbb{Z}$. Ann. Henri Poincar\'e (2025). https://doi.org/10.1007/s00023-024-01533-z

\bibitem{Sun2} J. Jian and Y. Sun, Dynamical localization for power-law long-range hopping random operators on ${\mathbb{Z}}^d$. Proc. Amer. Math. Soc. {\bf 150} (2022), 5369--5381.

\bibitem{Sunpartin} J. Jian and Y. Sun, Localization of interacting random particles with power-law long-range hopping. Ann. Henri Poincaré (2026). https://doi.org/10.1007/s00023-026-01726-8

\bibitem{Kerner} J. Kerner, O. Post, M. Sabri and M. Ta\"ufer, The curious spectra and dynamics of non-locally finite crystals. Comm. Math. Phys. {\bf 406} (2025), article number 169.
\bibitem{Klein-VonD}  H. Von Dreifus and A. Klein, A new proof of localization in the Anderson tight-binding model. {\bf 2} (1989) Comm. Math. Phys. 124, 285-299.

\bibitem{Kraisler} J. Kraisler, J. Schenker and J. C. Schotland, Dynamic one-photon localization in a discrete model of quantum optics. SIAM J. Appl. Math. {\bf 86} (2026), no. 3, 776--790.

\bibitem{Nazareno} H. N. Nazareno, C. A. A. da Silva, P. E. de Brito, Dynamical localization in aperiodic 1D systems under the action of electric fields. Superlattices Microstruct. {\bf 18}  (1995), 297--307.

\bibitem{Shi1} Y. Shi, A multi-scale analysis proof of the power-law localization for random operators on ${\mathbb{Z}}^d$. J. Differ. Equ. {\bf 297} (2021), 201-225.

\bibitem{Shi2} Y. Shi, Localization for almost-periodic operators with power-law long-range hopping: A Nash-Moser iteration type reducibility approach. Commun. Math. Phys. {\bf 402} (2023), 1765--1806.

\bibitem{Shi4} Y. Shi and L. Wen, Diagonalization in a quantum kicked rotor model with non-analytic potential. J. Differ. Equ. {\bf 355} (2023), 334--368.

\bibitem{Shi5}  Y. Shi and L. Wen, Green’s function estimates for quasi-periodic operators on $\mathbb{Z}^d$ with power-law long-range hopping. Adv. Math. {\bf 482} (2025), 110661.

\bibitem{Shi55} Y. Shi and L. Wen, Localization for a class of discrete long-range quasi-periodic operators. Lett. Math. Phys. {\bf 112} (2022), Article number 86.

\bibitem{Shi6} Y. Shi, L. Wen and D. Yan, Localization for random operators on $\mathbb{Z}^d$ with the long-range hopping. Ann. Henri Poincaré (2025). https://doi.org/10.1007/s00023-025-01595-7

\bibitem{Sun} Y. Sun and C. Wang, Localization of polynomial long-range hopping lattice operator with uniform electric fields. Lett. Math. Phys. {\bf 114} (2024), Article number 6.

\bibitem{Sun3} Y. Sun and C. Wang, Stark localization of Jacobi operator with applications to quantum spin models. Lett. Math. Phys. {\bf 115} (2025), Article number 126.

\bibitem{Tcheremchantsev} S. Tcheremchantsev, How to prove dynamical localization. Commun. Math. Phys. {\bf 221} (2001), 27--56.

\end{thebibliography}
\end{document}